\documentclass[sn-mathphys,Numbered]{sn-jnl}% Math and Physical Sciences Reference Style
\usepackage{graphicx}%
\usepackage{multirow}%
\usepackage{amsmath,amssymb,amsfonts}%
\usepackage{amsthm}%
\usepackage{mathrsfs}%
\usepackage[title]{appendix}%
\usepackage{xcolor}%
\usepackage{textcomp}%
\usepackage{manyfoot}%
\usepackage{booktabs}%
\usepackage{algorithm}%
\usepackage{algorithmicx}%
\usepackage{algpseudocode}%
\usepackage{listings}%
\theoremstyle{thmstyleone}%
\newtheorem{theorem}{Theorem}%  meant for continuous numbers
\theoremstyle{thmstyletwo}%
\newtheorem{corollary}{Corollary}%

\def\N{\mathbb{N}}
\def\R{\mathbb{R}}
\def\eq#1{{\rm(\ref{#1})}}

\begin{document}

\title[Hermite--Hadamard type inequalities by using Newton-Cotes quadrature formulas]{Hermite--Hadamard type inequalities by using Newton-Cotes quadrature formulas}


\author[1]{\fnm{Angshuman R.} \sur{Goswami}}\email{goswami.angshuman.robin@mik.uni-pannon.hu}

\author*[1]{\fnm{Ferenc} \sur{Hartung}}\email{hartung.ferenc@mik.uni-pannon.hu}

\affil*[1]{\orgdiv{Department of Mathematics}, \orgname{University of Pannonia}, \orgaddress{\street{Egyetem u. 10}, \city{Veszpr\'em}, \postcode{H-8200},  \country{Hungary}}}

\abstract{A convex function $f:[a,b]\to\R$ satisfies the so-called Hermite-Hadamard
inequality
$$
f\bigg(\dfrac{a+b}{2}\bigg)\leq \frac{1}{b-a}\int_a^{b}f(t)dt\leq \dfrac{f(a)+f(b)}{2}.
$$
Motivated by the above estimates, in this paper we consider approximately monotone and convex functions, and give upper and lower bounds to the numerical integral mean, i.e., to $\frac1{b-a}\mathcal{I}_{n}(f)$, where  $\mathcal{I}_{n}(f)$
denotes some of the most popular Newton-Cotes quadrature formulas.}




\maketitle

\section{Introduction}\label{sec1}

The theory of approximately convex functions goes back to the work of Hyers and Ulam \cite{HU} where they first introduced the concept of $\delta$-convexity. A function $f:I\to\R$ is said to be $\delta$-convex, if for any $x,y\in I$ and for all $t\in[0,1]$ the following functional inequality holds
$$
f(tx+(1-t)y)\leq tf(x)+(1-t)f(y)+\delta.
$$
They showed that a function satisfying $\delta$-convexity can be decomposed as the algebraic sum of an ordinary convex and a bounded function whose supremum norm is not greater than $\frac{\delta}{2}.$
Since then many different versions of approximate convexity were introduced and investigated (see \cite{GosPal20}, \cite{GosPal21}, \cite{890} and their references).

The notion of approximate monotonicity often appears during the study of generalized convex functions. A function $f:I\to\R$ is said to be $\varepsilon$-monotone, if it satisfies the following functional inequality
$$
f(x)\leq f(y)+\varepsilon \qquad\mbox{for all}\quad x,y\in I \quad\mbox{with}\quad x<y.
$$
P\'ales in \cite{MON} showed that an $\varepsilon$-monotone function $f$ can also be expressed as $g+h$; where $g$ is nondecreasing and $h$ is a bounded function satisfying $||h||\leq \frac{\varepsilon}{2}.$

We are going to see results associated with more generalized versions of approximate monotonicity and convexity. Throughout this paper $I\subseteq\R $ represents the non-empty and non-singleton interval $[a,b]$. The length of the interval $I$ will be denoted by $\ell(I)$. The function $\Phi:[0,\ell(I)]\to\R_+$ will be termed as error function, where $\R_+$ denotes the set of nonnegative reals.
With the help of error function
$\Phi$, we can introduce terminologies such as approximately monotone, H\"older, convex and affine functions. These function classes are studied in depth in the papers \cite{GosPal20}-\cite{890}, \cite{MakPal12a}, where along with structural properties, some characterizations, decompositions, several sandwich type theorems and applications are discussed.
For readability purposes, we recall these notions.

A function $f: I\to\R$ is said to be $\Phi$-monotone if for any $x,y\in I$ with $x<y$, the following inequality holds
$$f(x)\leq f(y)+\Phi(y-x).$$ This version of approximate monotonicity was first introduced in
\cite{GosPal20}. A more in depth study can be found in \cite{GosPal21}.
Due to association of the non-negative error term, the class of $\Phi$-monotone functions is bigger than the class of ordinary nondecreasing functions. If both of the functions $f$ and $-f$ satisfy $\Phi$-monotonicity, i.e.,
$$
|f(x)-f(y)|\leq \Phi(|y-x|)\quad \mbox{for all}\quad  x,y\in I,
$$
then $f$ is termed as $\Phi$-H\"older.

The following definition was introduced in \cite{MakPal12a}.
A function $f: I\to\R$ is said to be $\Phi$-convex if the  functional inequality
\begin{equation}\label{2}
 {f(tx+(1-t)y)\leq tf(x)+(1-t)f(y)+t\Phi((1-t)|y-x|)+(1-t)\Phi(t|y-x|)}
\end{equation}
holds for any $x,y\in I$ and for all $t\in[0,1]$.
It is also evident that any ordinary convex function also possesses $\Phi$-convexity.
If both $f$ and $-f$ are $\Phi$-convex, i.e.,
$$
|f(tx+(1-t)y)-tf(x)-(1-t)f(y)|\leq t\Phi((1-t)|y-x|)+(1-t)\Phi(t|y-x|)
$$
holds for any $x,y\in I$ and for all $t\in[0,1]$, then we say that $f$ is
a $\Phi$-affine function.

In \cite{GosPal21}, one of the results states that if $f:I\to\R$ is $\Phi$-monotone, then
$$
f(a)-\dfrac{1}{b-a}\int_0^{b-a}\Phi(s)\,ds\leq \dfrac{1}{b-a}\int_a^{b}f(s)\,ds
\leq f(b)+\dfrac{1}{b-a}\int_0^{b-a}\Phi(s)\,ds
$$
holds,
provided both $f$ and $\Phi$ are Lebesgue integrable.
Our goal in this paper is to estimate the numerical integral mean of $f$, i.e.,
$\frac1{b-a}\mathcal{I}_{n}(f)$, where  $\mathcal{I}_{n}(f)$
denotes some of the most popular Newton-Cotes quadrature formulas: the Trapezoidal, Simpson's
and Simpson's $3/8$ rules. Without assuming the Lebesgue integrability of the error function $\Phi$, we can show that a continuous $\Phi$-monontone function $f$ satisfies the following inequality
$$
f(a)-\dfrac{n}{2}\Phi\bigg(\dfrac{b-a}{n}\bigg)\leq \dfrac{1}{b-a}\mathcal{I}_{n}(f)
\leq f(b)+\dfrac{n}{2}\Phi\bigg(\dfrac{b-a}{n}\bigg).
$$
Let $x_i=a+\frac{(b-a)i}{n}$, $i=0,1,\ldots,n$ be an equidistant partition of the interval $[a,b]$.
For a continuous function $f:I\to\R$, we recall the Trapezoidal, Simpson's
and Simpson's $3/8$
rules as follows (see e.g. \cite{BF}):
$$\mathcal{T}_{n}(f):=\bigg(\dfrac{b-a}{2n}\bigg)\bigg[f(x_0)+2\sum_{i=1}^{n-1}f(x_i)+f(x_n)\bigg] \quad (n\in\N),$$
$$\mathcal{S}_{n}%^{\frac{1}{3}}
(f):=\bigg(\dfrac{b-a}{3n}\bigg)
\bigg[f(x_0)+2\sum_{i=1}^{(n/2)-1}f(x_{2i})+4\sum_{i=1}^{n/2}f(x_{2i-1})+f(x_n)\bigg]
\quad (n\in\N \mbox{ is even}),$$
$$\mbox{and}$$
$$\mathcal{S}_{n}^{\frac{3}{8}}(f):=\bigg(\dfrac{3(b-a)}{8n}\bigg)
\sum_{i=1}^{n/3}%\frac{n}{3}
\Big[f(x_{3i-3})+3f(x_{3i-2})+3f(x_{3i-1})+f(x_{3i})\Big]
\ \ (n=3k,\  k\in\N).$$

The classical Hermite-Hadamard inequality was independently introduced by Hermite and Hadamard in the \cite{Hadamard} and \cite {Hermite}. It states that "In a compact interval, integral mean of a convex function is greater than the functional value at the mid-point, while the same is dominated by the average of the functional values at the extreme points of that interval." Mathematically, for a convex function $f:I\to\R$, this inequality can be represented as follows
\begin{equation}\label{eHH}
f\bigg(\dfrac{a+b}{2}\bigg)\leq \frac1{b-a}\int_a^{b}f(t)dt\leq \dfrac{f(a)+f(b)}{2}.
\end{equation}
Various generalized forms of this inequality were presented for approximate and higher order convex functions. Along with it many associated concepts and applications were also studied (see e.g. \cite{Bessenyei1}, \cite{Bessenyei2}, \cite{Dragomir} and the references therein).

In this paper, we present the numerical version of Hermite-Hadamard inequality for the continuous $\Phi$-convex function $f:I\to\R$. We show that for any even number $n\in\N$, the following functional inequality holds
$$
f\bigg(\dfrac{a+b}{2}\bigg)-\dfrac{n^2+2}{12}\Phi\Big(\frac{b-a}{n}\Big)\leq
\dfrac{1}{b-a}\mathcal{T}_{n}(f)\leq \dfrac{f(a)+f(b)}{2}+\dfrac{n^2-1}{6}\Phi\Big(\frac{b-a}{n}\Big).
$$
One can easily observe that it resembles very closely with the original inequality \eq{eHH}.
We show similar estimates for odd $n$, see Theorem~\ref{T333} below.

For the cases of $\Phi$-H\"older and $\Phi$-affine functions we will obtain inequalities with improved lower and upper bounds.
Furthermore, in our results, we discuss particular cases where we can omit the dependency on $n$ from the error terms.

The graphs of approximately monotone and convex functions  often appear in the stock prices of share market, various financial and business models, population growth of certain regions etc. To have an estimate of mean values for such scenarios, numerical integration techniques are  useful. Through this we can neglect the stochastic nature of the associated error term.

\section{On approximately monotone functions}\label{sec2}

We start with $\Phi$-monotone and $\Phi$-H\"older functions.
As discussed before, throughout this section, $I$ will denote the non-empty and non-singleton interval $[a,b]$ with $a<b$.
\begin{theorem}\label{T1}
Let $f:I\to\R$ be a continuous $\Phi$-monotone function. Then for any $n\in\N$, the following inequalities hold
\begin{equation}\label{T}
f(a)-\frac{n}{2}\Phi\Big(\frac{b-a}{n}\Big)\leq\frac{1}{b-a}\mathcal{T}_{n}(f)\leq f(b)+ \frac{n}{2}\Phi\Big(\frac{b-a}{n}\Big).
\end{equation}
\end{theorem}
\begin{proof}
For the proof, we consider partitioning the interval $I$ into $n$ equal sub-intervals as follows
\begin{equation}\label{part}
P=\{ a=x_0<x_1<\cdots <x_n=b\}\quad\mbox{with}\quad x_i-x_{i-1}=\frac{b-a}{n},\quad i\in\{1,\ldots,n\}.
\end{equation}
By applying $\Phi$-monotonicity in each of the subintervals $[x_{i-1},x_i]$ of $I$, we get the following system of $n$ inequalities
\begin{equation}\label{q1}
\begin{aligned}
f(x_0)-f(x_1)&\leq  \Phi\Big(\frac{b-a}{n}\Big)\\
&\vdots\\
f(x_{n-1})-f(x_n)&\leq \Phi\Big(\frac{b-a}{n}\Big).
\end{aligned}
\end{equation}
First multiplying each inequalities of the above system respectively by
$1,3,5,\ldots,2n-1$
and then summing those up, we arrive at
\begin{align*}
f(x_0)+2\sum_{i=1}^{n-1}f(x_i)+f(x_n)
&\leq 2nf(x_n)+\bigg(\sum_{i=1}^{n}(2i-1)\bigg)\Phi\Big(\frac{b-a}{n}\Big)\\
&=2nf(x_n)+n^2\Phi\Big(\frac{b-a}{n}\Big).
\end{align*}
Now multiplying both side of the above inequality by $\dfrac{b-a}{2n}$ and using the definition of the Trapezoidal
rule, we obtain
$$
\mathcal{T}_{n}(f)\leq (b-a)\bigg(f(b)+ \frac{n}{2}\Phi\Big(\frac{b-a}{n}\Big)\bigg).
$$
Dividing both side by  $b-a$, we get the second part of inequality \eq{T}.

Similarly, multiplying each inequalities of \eq{q1} by
$2n-1,2n-3,\ldots,1$, respectively,
and then adding the resultant inequalities, we arrive at
$$
(2n)f(x_0)-\bigg(\sum_{i=1}^{n}\big(2(n-i)+1\big)\bigg)\Phi\Big(\frac{b-a}{n}\Big)\leq f(x_0)+2\sum_{i=1}^{n-1}f(x_i)+f(x_n).
$$
As before, we multiply the both sides of this inequality by $\dfrac{b-a}{2n}$ and utilizing the definition of
the Trapezoidal rule, we get
$$
(b-a)\bigg(f(a)- \frac{n}{2}\Phi\Big(\frac{b-a}{n}\Big)\bigg)\leq\mathcal{T}_{n}(f).
$$
This yields the first part of inequality \eq{T} and establishes
the statement.
\end{proof}
Next we are going to see result associated with trapezoidal integral mean for the $\Phi$-H\"older functions.
\begin{corollary}\label{C1}
Let $f:I\to\R$ be a continuous $\Phi$-H\"older function. Then for any $n\in\N$, the following inequalities hold
\begin{equation}\label{T5}
\max\{f(a),f(b)\}-\frac{n}{2}\Phi\Big(\frac{b-a}{n}\Big)\leq\frac{1}{b-a}\mathcal{T}_{n}(f)\leq \min\{f(a),f(b)\}+ \frac{n}{2}\Phi\Big(\frac{b-a}{n}\Big).
\end{equation}
\end{corollary}
\begin{proof}
For the proof, we consider partitioning the interval $I$ as defined in \eq{part}. Since $f$ is $\Phi$-H\"older;
the function $(-f)$ also possesses $\Phi$-monotonicity. Therefore by Theorem~\ref{T1}, substituting $-f$ instead of $f$ in \eq{T}, we obtain
$$
f(b)-\frac{n}{2}\Phi\Big(\frac{b-a}{n}\Big)\leq\frac{1}{b-a}\mathcal{T}_{n}(f)\leq f(a)+ \frac{n}{2}\Phi\Big(\frac{b-a}{n}\Big).
$$
This together with \eq{T} yields the inequality \eq{T5} and completes the proof.
\end{proof}
The next corollary shows that under a usual characteristics of error function $\Phi$, the dependency of $n$ can be relaxed in the bounds of the numerical integral. We say that the error function $\Phi$ is superadditive if
$\Phi(x)+\Phi(y)\leq \Phi(x+y)$ holds for any $x,y$ and $x+y\in [0,\ell(I)]$.
\begin{corollary}\label{C717}
Let $f:I\to\R$ be a continuous $\Phi$-monotone function where $\Phi$ possesses supperadditivity. Then for any $n\in\N$, the following inequalities hold
\begin{equation}\label{T56}
f(a)-\dfrac{1}{2}\Phi(b-a)\leq \dfrac{1}{b-a}\mathcal{T}_{n}(f)\leq f(b)+\dfrac{1}{2}\Phi(b-a).
\end{equation}
\end{corollary}
\begin{proof}
Since $\Phi:[0,\ell(I)]\to \R_+$ is supperadditive, for any $n\in\N$, it satisfies the following functional inequality
\begin{equation}\label{8765}
n\Phi\bigg(\dfrac{b-a}{n}\bigg)=\sum_{i=1}^{n}\Phi\bigg(\dfrac{b-a}{n}\bigg)\leq \Phi\bigg(\sum_{i=1}^{n}\dfrac{b-a}{n}\bigg)=\Phi(b-a).
\end{equation}
Using the above inequality at \eq{T} of Theorem~\ref{T1}, we obtain the desired result.
\end{proof}
The above corollary can also be be formulated for Theorems~\ref{T2} and \ref{T3} under necessary assumptions on $n$.
The next theorems replicate the similar result for Simpson's
and Simpson's $3/8$
rules.
\begin{theorem}\label{T2}
Let $f:I\to\R$ be a continuous a $\Phi$-monotone function. Then for any even number $n\in \N$, the following inequalities hold
\begin{equation}\label{TTTTTT}
f(a)-\frac{n}{2}\Phi\Big(\frac{b-a}{n}\Big)\leq\frac{1}{b-a}\mathcal{S}_{n}
(f)\leq f(b)+ \frac{n}{2}\Phi\Big(\frac{b-a}{n}\Big).
\end{equation}
\end{theorem}
\begin{proof}
To prove the statement, we consider partitioning the interval $I$ into $n$ equal sub-intervals as mentioned in \eq{part}.
Utilizing $\Phi$-monotonicity of $f$ in each of the subintervals $[x_{i-1},x_i]$ of $I$, we get the system of $n$ inequalities as mentioned in \eq{q1}.

Now we multiply both sides of the odd numbered inequalities of \eq{q1}
by $(6i-5)$ and even numbered inequalities  by $(6i-1)$ respectively for $i=1,\ldots,\frac{n}{2}$ and summing up those side by side, we obtain
$$
\sum_{i=1}^{n/2}
\big[f(x_{2i-2})+4f(x_{2i-1})+f(x_{2i})\big]\leq 3nf(x_n)+\Phi\Big(\dfrac{b-a}{n}\Big)\Bigg(\sum_{i=1}^{n/2}(6i-1)+\sum_{i=1}^{n/2}(6i-5)\Bigg).
$$
It can be easily observable that
$\displaystyle\sum_{i=1}^{n/2}(6i-1)+\displaystyle\sum_{i=1}^{n/2}(6i-5)=\dfrac{3}{2}n^2$.
Upon multiplying  the above inequality by $\dfrac{b-a}{3n}$ and using the definition of Simpson's
rule we get
$$
\mathcal{S}_{n}
(f)\leq (b-a)\bigg(f(b)+ \frac{n}{2}\Phi\Big(\frac{b-a}{n}\Big)\bigg).
$$
Dividing both side of it by  $b-a$, we
obtain the right most inequality of \eq{TTTTTT}.

To show the initial part of inequality \eq{TTTTTT}, once again we consider the system of inequalities \eq{q1}.
First we multiply both sides of odd numbered inequalities by
$6(\frac{n}{2}-i)+5$
and then even numbered inequalities by
$6(\frac{n}{2}-i)+1$
for
$i=1,2,\ldots,\frac{n}{2}$ respectively.
After adding up the resultant inequalities side by side, we arrive at
\begin{align*}
3nf(x_0)-\Phi\Big(\dfrac{b-a}{n}\Big)\Bigg(\sum_{i=1}^{n/2}\Big(6\Big(\frac{n}{2}-i\Big)+1\Big)+&\sum_{i=1}^{n/2}\Big(6\Big(\frac{n}{2}-i\Big)+5\Big)\Bigg)\\
&\leq
\sum_{i=1}^{n/2}\big[f(x_{2i-2})+4f(x_{2i-1})+f(x_{2i})\big].
\end{align*}
Multiplying the above inequality by $\dfrac{b-a}{3n}$ and using the definition of Simpson's
rule we obtain
$$
(b-a)\bigg(f(a)-\frac{n}{2}\Phi\Big(\frac{b-a}{n}\Big)\bigg)\leq\mathcal{S}_{n}
(f).
$$
This shows the first inequality part of \eq{TTTTTT} and establishes the result.
\end{proof}
The corollary below can be derived by using the above theorem. Hence the proof of it is not included.
\begin{corollary}\label{C2}
Let $f:I\to\R$ be a continuous $\Phi$-H\"older function. Then for any even number $n\in\N$, the following inequalities hold
$$
\max\{f(a),f(b)\}-\frac{n}{2}\Phi\Big(\frac{b-a}{n}\Big)\leq\frac{1}{b-a}\mathcal{S}_{n}
(f)\leq \min\{f(a),f(b)\}+ \frac{n}{2}\Phi\Big(\frac{b-a}{n}\Big).
$$
\end{corollary}
The next corollary shows that in case the error function $\Phi$ is superadditive,
then we can obtain
bounds for the numerical integral mean of $f$, which are
independent of $n$.
\begin{corollary}\label{C7117}
Let $f:I\to\R$ be a continuous $\Phi$-H\"older function where $\Phi$ possesses supperadditivity.
Then
$$
\max\{f(a),f(b)\}-\dfrac{1}{2}\Phi(b-a)\leq \dfrac{1}{b-a}\mathcal{S}_{n}
(f)\leq \min\{f(a),f(b)\}+\dfrac{1}{2}\Phi(b-a)
$$
hold for any even number $n\in \N$.
\end{corollary}
\begin{proof}
The statement is a direct consequence of Corollary~\ref{C2} and \eq{8765}.
\end{proof}
We can construct similar results for $\Phi$-H\"older functions for the Trapezoidal and Simpson's
$3/8$ rule under the appropriate restriction/relaxation on the value of $n$.
\begin{theorem}\label{T3}
Let $f:I\to\R$ be a continuous $\Phi$-monotone function. Then for any number $n\in \N$
which is a multiple of 3, the inequalities
\begin{equation}\label{SS}
		f(a)-\frac{n}{2}\Phi\Big(\frac{b-a}{n}\Big)\leq\frac{1}{b-a}\mathcal{S}_{n}^{\frac{3}{8}}(f)dx\leq f(b)+ \frac{n}{2}\Phi\Big(\frac{b-a}{n}\Big)
\end{equation}
hold.
\end{theorem}
\begin{proof}
To prove the statement, we consider partitioning the interval $I$ into $n$ equal sub-intervals as mentioned in \eq{part}.
Utilizing $\Phi$-monotonicity of $f$ in each of the subintervals $[x_{i-1},x_i]$ of $I$, we get the system of $n$ inequalities as mentioned in \eq{q1}.
Since
$n$ is a multiple of 3,
we orderly rearrange the system of inequalities in $\dfrac{n}{3}$ triplets as follows
\begin{equation}\label{090}\begin{aligned}
f(x_{3k-3})&\leq f(x_{3k-2})+\Phi\Big(\dfrac{b-a}{n}\Big)\\
f(x_{3k-2})&\leq f(x_{3k-1})+\Phi\Big(\dfrac{b-a}{n}\Big)\quad\quad\quad\quad \bigg(k=1,2,\ldots,n/3\bigg)\\
f(x_{3k-1})&\leq f(x_{3k})+\Phi\Big(\dfrac{b-a}{n}\Big).
\end{aligned}\end{equation}
Now initially, we multiply the first, second and 3rd inequalities of $k^{th}$ triplet respectively by $8k-7$, $8k-4$ and $8k-1$. After summing up all resultant inequalities side by side, we arrive at the following
\begin{align*}
\sum_{k=1}^{n/3}\Big((8k-7)f(x_{3k-3})&+3f(x_{3k-2})+3f(x_{3k-1})-(8k-1)f(x_{3k})\Big)\\
&\leq\bigg(\sum_{k=1}^{n/3}(24k-12)\bigg)\Phi\Big(\dfrac{b-a}{n}\Big).
\end{align*}
One can easily see that $\displaystyle\sum_{k=1}^{n/3}(24k-12)=\dfrac{4}{3}n^2.$
Now rewriting the above inequality, we get
$$
\sum_{i=1}^{n/3}\Big(f(x_{3i-3})+3f(x_{3i-2})+3f(x_{3i-1})+f(x_{3i})\Big)
\leq \dfrac{8n}{3}f(b)+\dfrac{4}{3}n^2\Phi\Big(\dfrac{b-a}{n}\Big).
$$
Multiplying both sides of the above inequality by $\dfrac{3(b-a)}{8n}$ and using the definition of Simpson's
$3/8$ rule, we obtain the following functional inequality
$$
\mathcal{S}_{n}^{\frac{3}{8}}(f)\leq (b-a)\bigg(f(b)+ \frac{n}{2}\Phi\Big(\frac{b-a}{n}\Big)\bigg).
$$
By diving both sides of it by $(b-a)$ we get the right most inequality of \eq{SS}.

To obtain the first part of the inequality, we
multiplying the first, second and 3rd inequalities of $k^{th}$ triplet in \eq{090} by $8(n/3-k)+7$, $8(n/3-k)+4$ and $8(n/3-k)+1$ respectively for $k=1,\ldots ,n/3$. After that adding up all the inequalities side by side, we arrive at the following
$$
\dfrac{8n}{3}f(a)-\dfrac{4}{3}n^2{\Phi\Big(\dfrac{b-a}{n}\Big)}\leq
\sum_{i=1}^{n/3}\Big[f(x_{3i-3})+3f(x_{3i-2})+3f(x_{3i-1})+f(x_{3i})\Big].
$$
Now we multiply  both sides of it by $\dfrac{3(b-a)}{8n}$ and using the definition of Simpson's
{$3/8$} rule, we obtain the functional inequality below
$$
(b-a)\bigg(f(a)- \frac{n}{2}\Phi\Big(\frac{b-a}{n}\Big)\bigg)\leq\mathcal{S}_{n}^{\frac{3}{8}}(f).
$$
This yields the first part of inequality \eq{SS} and
completes the proof.
\end{proof}
The establishment of the following result is a direct implication of the above theorem. For that reason, we just state the statement.
\begin{corollary}\label{C3}
Let $f:I\to\R$ be a continuous $\Phi$-H\"older function. Then for any number $n\in \N$ which is a multiple of 3, the function $f$ satisfies the following functional inequalities
$$
\max\{f(a),f(b)\}-\frac{n}{2}\Phi\Big(\frac{b-a}{n}\Big)\leq\frac{1}{b-a}\mathcal{S}_{n}^{\frac{3}{8}}(f)\leq \min\{f(a),f(b)\}+ \frac{n}{2}\Phi\Big(\frac{b-a}{n}\Big).
$$
\end{corollary}
In the next section, we discuss  results related to $\Phi$-convex and $\Phi$-affine functions.

\section{On approximately convex functions}\label{sec4}

The main objective of this section is to present a numerical version of Hermite-Hadamard type inequality for $\Phi$-convex functions. Here also, $I$ will stand for non-empty and non-singleton interval $[a,b]$ with $a<b$.
\begin{theorem}\label{T333}
Suppose $f:I\to\R$ is a continuous $\Phi$-convex function. Then for any $n\in\N$, the following inequalities satisfy
\begin{equation}\label{SS2}
f\bigg(\dfrac{a+b}{2}\bigg)-E_n\leq\frac{1}{b-a}\mathcal{T}_{n}(f)\leq \dfrac{f(a)+f(b)}{2}+\dfrac{n^2-1}{6}\Phi\Big(\frac{b-a}{n}\Big),
\end{equation}
\end{theorem}
where
\begin{equation}\label{SS2b}
E_n=
\left\{\begin{array}{ll}
   \frac{n^2+2}{12}\Phi\!\left(\frac{b-a}{n}\right), & \mbox{if $n$ is even,}\\[0.1cm]
   \frac{n^2-1}{12}\Phi\!\left(\frac{b-a}{n}\right)+\Phi\!\left(\frac{b-a}{2n}\right), &
   \mbox{if $n$ is odd.}
       \end{array}\right.
\end{equation}
\begin{proof}
Let $x<y$.
By substituting $t=\frac{1}{2}$  in \eq{2}, we arrive at the following
\begin{equation}\label{10}
2f\bigg(\dfrac{x+y}{2}\bigg)\leq f(x)+f(y)+2\Phi\bigg(\dfrac{y-x}{2}\bigg).
\end{equation}
We partition the interval $I$ as defined in \eq{part}.
Now instead of $x$ and $y$; substituting the pairs $(x_{i-1},x_{i+1})$ for all $i=1,\ldots,n-1$ in \eq{10}, we get the following system of $n-1$ inequalities
\begin{equation}\label{sys1}\begin{aligned}
2f(x_1)=2f\bigg(\dfrac{x_0+x_2}{2}\bigg)&\leq f(x_0)+f(x_2)+2\Phi\bigg(\dfrac{b-a}{n}\bigg)\\
&\vdots\\
2f(x_{n-1})=2f\bigg(\dfrac{x_{n-2}+x_n}{2}\bigg)&\leq f(x_{n-2})+f(x_n)+2\Phi\bigg(\dfrac{b-a}{n}\bigg)
\end{aligned}\end{equation}
To establish \eq{SS2}, we consider two cases.  At first we assume that $n$ is odd. Now summing up the all $n-1$ inequalities of \eq{sys1}, we obtain the following
$$
f(x_1)+f(x_{n-1})\leq f(x_0)+f(x_n)+2(n-1)\Phi\bigg(\dfrac{b-a}{n}\bigg).
$$
Now excluding the very first and last inequalities from \eq{sys1} and summing up all the remaining $n-3$ inequalities, we arrive at
$$
f(x_{2})+f(x_{n-2})\leq f(x_1)+f(x_{n-1})+2(n-3)\Phi\bigg(\dfrac{b-a}{n}\bigg).
$$
Continuing the same way, we finally have the system of $\dfrac{n-1}{2}$ inequalities as follows
\begin{equation}\label{eq10}\begin{aligned}
f\big(x_{\frac{n-1}{2}}\big)+f\big(x_{\frac{n+1}{2}}\big)&\leq f\big(x_{\frac{n-3}{2}}\big)+f\big(x_{\frac{n+3}{2}}\big)+4\Phi\bigg(\dfrac{b-a}{n}\bigg)\\
f\big(x_{\frac{n-3}{2}}\big)+f\big(x_{\frac{n+3}{2}}\big)&\leq f\big(x_{\frac{n-5}{2}}\big)+f\big(x_{\frac{n+5}{2}}\big)+8\Phi\bigg(\dfrac{b-a}{n}\bigg)\\
&\vdots\\
f(x_2)+f(x_{n-2})&\leq f(x_1)+f(x_{n-1})+2(n-3)\Phi\bigg(\dfrac{b-a}{n}\bigg)\\
f(x_1)+f(x_{n-1})&\leq f(x_0)+f(x_n)+2(n-1)\Phi\bigg(\dfrac{b-a}{n}\bigg).
\end{aligned}\end{equation}
Now multiplying each of the inequalities of the above system by $2,4,\ldots,n-3,n-1$ respectively and summing those up side by side we have
$$
2[f(x_1)+\ldots +f(x_{n-1})]\leq (n-1)(f(x_0)+f(x_n))+2\bigg(\sum_{i=1}^{(n-1)/{2}}(2i)^2\bigg)\Phi\bigg(\dfrac{b-a}{n}\bigg).
$$
Adding $f(x_0)+f(x_n)$ to both sides of the above inequality and multiplying the resulting inequality by $\dfrac{b-a}{2n}$, we obtain
\begin{equation}\label{22222}
\mathcal{T}_{n}(f)\leq (b-a)\bigg[\dfrac{f(a)+f(b)}{2}+\dfrac{n^2-1}{6}\Phi\Big(\frac{b-a}{n}\Big)\bigg].
\end{equation}
On the other hand, if $n$ is even; we formulate the following system of $n-1$ inequalities by
applying the same method as described above
\begin{equation}\label{sys12}\begin{aligned}
2f\big(x_{\frac{n}{2}}\big)&\leq f\big(x_{\frac{n}{2}-1}\big)+f\big(x_{\frac{n}{2}+1}\big)+2\Phi\bigg(\dfrac{b-a}{n}\bigg)\\
f\big(x_{\frac{n}{2}-1}\big)+f\big(x_{\frac{n}{2}+1}\big)&\leq f\big(x_{\frac{n}{2}-2}\big)+f\big(x_{\frac{n}{2}+2}\big)+6\Phi\bigg(\dfrac{b-a}{n}\bigg)\\
&\vdots\\
f(x_2)+f(x_{n-2}) &\leq f(x_1)+f(x_{n-1})+2(n-3)\Phi\bigg(\dfrac{b-a}{n}\bigg)\\
f(x_1)+f(x_{n-1})&\leq f(x_0)+f(x_n)+2(n-1)\Phi\bigg(\dfrac{b-a}{n}\bigg).
\end{aligned}\end{equation}
We multiply each inequalities of the above system respectively by $1,3,\ldots,n-3, n-1$ and then adding the resultant inequalities side by side, we arrive the following
$$
2[f(x_1)+\cdots +f(x_{n-1})]\leq (n-1)(f(x_0)+f(x_n))+2\bigg(\sum_{i=1}^{n/{2}}(2i-1)^2\bigg)
\Phi\bigg(\dfrac{b-a}{n}\bigg).
$$
By adding $f(x_0)+f(x_n)$, and then  multiplying both sides of the  resulting  inequality by $\dfrac{b-a}{2n}$ and utilizing the identity $2\displaystyle\sum_{i=1}^{n/{2}}(2i-1)^2=\frac{n^3-n}{3}$  we again obtain \eq{22222}.
This yields  the second inequality of \eq{SS2} for any $n\in\N$ and establishes the first assertion.

To show the the first inequality of \eq{SS2}, first we assume $n$ is even, and we construct the $n-1$ system of inequalities as described in \eq{sys12}. Now first multiplying the inequalities of the system by $n-1,n-3,\ldots,1$ respectively and then summing the resultant inequalities side by side, and using the identity
$\displaystyle2\sum_{i=1}^{n/2}(2i-1)(n-(2i-1))=\frac{n(n^2+2)}{6}$,
we obtain
$$
2nf(x_{\frac{n}{2}})\leq f(x_0)+2\sum_{i=1}^{n-1}f(x_i)+f(x_n)
+\dfrac{n(n^2+2)}{6}\Phi\bigg(\dfrac{b-a}{n}\bigg).
$$
Now we again multiply both sides of the above inequality by $\dfrac{b-a}{n}$. Applying the trapezoidal rule to the resultant, we get
$$
(b-a)\bigg[f\bigg(\frac{a+b}{2}\bigg)-\dfrac{n^2+2}{12}\Phi\bigg(\dfrac{b-a}{n}\bigg)\bigg]\leq \mathcal{T}_{n}(f).
$$

Next we suppose $n$ is odd.  We multiply the inequalities in \eq{eq10} by $n-2,n-4,\ldots,3,1$, respectively,
and add them up, then we get
\begin{align*}
 (n-2)\left[f\big(x_{\frac{n-1}{2}}\big)+f\big(x_{\frac{n+1}{2}}\big)\right]&\leq
f(x_0)+2\biggl(f(x_1)+\cdots+f\big(x_{\frac{n-3}{2}}\big)+f\big(x_{\frac{n+3}{2}}\big)+\cdots\\
&\quad+f(x_{n-1})\biggr)+f(x_n)+2 \sum_{i=1}^{(n-1)/2} 2i(n-2i)
\Phi\bigg(\dfrac{b-a}{n}\bigg).
\end{align*}
This yields
$$
n\left[f\big(x_{\frac{n-1}{2}}\big)+f\big(x_{\frac{n+1}{2}}\big)\right]\leq
f(x_0)+2\sum_{i=1}^{n-1}f(x_i)+f(x_n)+\dfrac{n^3-n}{6}\Phi\bigg(\dfrac{b-a}{n}\bigg),
$$
hence by multiplication with $\dfrac{1}{2n}$ we obtain
$$
\dfrac{1}{2}\left[f\big(x_{\frac{n-1}{2}}\big)+f\big(x_{\frac{n+1}{2}}\big)\right]\leq
\dfrac1{b-a}\mathcal{T}_{n}(f)+\dfrac{n^2-1}{12}\Phi\bigg(\dfrac{b-a}{n}\bigg).
$$
Then \eq{10} implies
\begin{align*}
 f\left(\dfrac{a+b}2 \right)-\Phi\bigg(\dfrac{b-a}{2n}\bigg)&\leq \dfrac{1}{2}\left[f\big(x_{\frac{n-1}{2}}\big)+f\big(x_{\frac{n+1}{2}}\big)\right]\\
&\leq
\dfrac1{b-a}\mathcal{T}_{n}(f)+\dfrac{n^2-1}{12}\Phi\bigg(\dfrac{b-a}{n}\bigg).
\end{align*}
This together with \eq{SS2b} establishes the first inequality of \eq{SS2} and completes the proof.
\end{proof}
For the case $\Phi\equiv 0$; the definition of $\Phi$-convex function (in equation \eq{2}), turns to be  the usual convexity. From the above theorem; it is  evident that if $f:I\to\R$ is convex, then for any  $n\in\N$, the function $f$ satisfies the following inequality
$$
f\bigg(\frac{a+b}{2}\bigg)\leq\dfrac{1}{b-a} \mathcal{T}_{n}(f)\leq\dfrac{f(a)+f(b)}{2}.
$$
Next, we are going to study Hermite-Hadamard inequality for $\Phi$-affine functions.
\begin{corollary}\label{T3333}
Suppose $f:I\to\R$ is a continuous $\Phi$-affine function. Then for any $n\in\N$, the following
inequalities are satisfied
\begin{align}
\max\bigg\{f\bigg(\dfrac{a+b}{2}\bigg)&-E_n,\,\,\dfrac{f(a)+f(b)}{2}-\dfrac{n^2-1}{6}\Phi\Big(\frac{b-a}{n}\Big)\bigg\}\nonumber\\
&\leq\frac{1}{b-a}\mathcal{T}_{n}(f)\label{SS9}\\
&\leq\min\bigg\{f\bigg(\dfrac{a+b}{2}\bigg)+E_n,\,\,\dfrac{f(a)+f(b)}{2}+\dfrac{n^2-1}{6}\Phi\Big(\frac{b-a}{n}\Big)\bigg\},\quad\nonumber
\end{align}
where $E_n$ is defined by \eq{SS2b}.
\end{corollary}
\begin{proof}
Since $f$ is $\Phi$-affine, together with $f$, the function $-f$  also satisfies $\Phi$-convexity. Thus instead of $f$, we can replace $-f$ in \eq{SS2} to obtain
$$
		\dfrac{f(a)+f(b)}{2}-\dfrac{n^2-1}{6}\Phi\Big(\frac{b-a}{n}\Big)\leq \frac{1}{b-a}\mathcal{T}_{n}(f)
		\leq f\bigg(\dfrac{a+b}{2}\bigg)+E_n.
$$
This inequality together with \eq{SS2} establishes the the result.
\end{proof}

We note that by following the method of Theorem~\ref{T333}
 we can obtain similar estimates  for Simpson's and Simpson's $3/8$ rules.
Clearly, the proofs of this paper can be easily extended to estimates of the numerical integral mean
associated to other Newton-Cotes quadrature formulas.

We close this paper with the following observation. If $f\in C^2[a,b]$, then it is known (see, e.g., \cite{BF}) that
$\mathcal{T}_{n}(f)\to \int_a^b f(x)\,ds$ as $n\to\infty$. Therefore, if
the error function satisfies
$$
\lim_{n\to\infty} n^2\Phi\Big(\frac{b-a}{n}\Big)=0,
$$
then  estimates \eq{SS2} imply the classical Hermite-Hadamard inequality \eq{eHH} for $\Phi$-convex functions too.

\backmatter

\section*{Declarations}

\bmhead{Funding}
FH thanks the support  of the Hungarian National Research, Development and Innovation Office grant no. K139346.

\bmhead{Competing interests}
The authors declare no competing interests.

%\bibliography{JIA_Goswami_Hartung}% common bib file
%\input{JIA_Goswami_Hartung.bbl}

\end{document}